\documentclass[11pt, a4paper]{article}
\usepackage[margin=3.2cm]{geometry}

\usepackage[inline]{showlabels}

\usepackage[utf8]{inputenc}
\usepackage[english]{babel}
\usepackage{amsmath}
\usepackage{amsfonts}
\usepackage{amsthm}
\usepackage{graphicx}
\usepackage{physics}
\usepackage{mathtools}
\usepackage{xspace}
\usepackage{tikz-cd}
\usepackage{hyperref}
\usepackage[toc,page]{appendix}

\usepackage{enumerate}
\usepackage{enumitem}

\theoremstyle{definition}
\newcounter{counter}
\numberwithin{counter}{section}
\newtheorem{theorem}[counter]{Theorem}
\newtheorem{proposition}[counter]{Proposition}
\newtheorem{definition}[counter]{Definition}
\newtheorem{lemma}[counter]{Lemma}
\newtheorem{corollary}[counter]{Corollary}
\newtheorem{example}[counter]{Example}
\newtheorem*{theorem*}{Theorem}
\newtheorem*{lemma*}{Lemma}
\newtheorem{remark}[counter]{Remark}
\newtheorem*{remark*}{Remark*}

\newcommand{\R}{\mathbb{R}}

\newcommand{\N}{\mathbb{N}}

\newcommand{\id}{\text{id}}

\newcommand{\sa}{\text{sa}}

\newcommand{\mult}{\,\&\,}
\newcommand{\commu}{\,\lvert\,}

\newcommand{\sse}{\subseteq}

\makeatletter
\newcommand\etc{etc\@ifnextchar.{}{.\@}\xspace}
\newcommand\ie{i.e.\@\xspace}  

\makeatother

\newcommand\Define[1]{\textbf{#1}}

\begin{document}

\title{Commutativity in Jordan Operator Algebras}
\author{John van de Wetering\\[0.3cm]
Radboud University Nijmegen, Netherlands\\[0.3cm]
\texttt{john@vdwetering.name}}

\date{\today}

\maketitle

\begin{abstract}
While Jordan algebras are commutative, their non-associativity makes it so that the Jordan product operators do not necessarily commute. When the product operators of two elements commute, the elements are said to \emph{operator commute}. In some Jordan algebras operator commutation can be badly behaved, for instance having elements $a$ and $b$ operator commute, while $a^2$ and $b$ do not operator commute. In this paper we study JB-algebras, real Jordan algebras which are also Banach spaces in a compatible manner, of which C$^*$-algebras are examples. We show that elements $a$ and $b$ in a JB-algebra operator commute if and only if they span an associative sub-algebra of mutually operator commuting elements, and hence operator commutativity in JB-algebras is as well-behaved as it can be. Letting $Q_a$ denote the quadratic operator of $a$, we also show that positive $a$ and $b$ operator commute if and only if $Q_a b^2 = Q_b a^2$. We use this result to conclude that the unit interval of a JB-algebra is a sequential effect algebra as defined by Gudder and Greechie.
\end{abstract}

\section{Introduction}

A Jordan algebra $(E,*,1)$ is a commutative bilinear unital algebra satisfying the Jordan equation $(a*b)*a^2 = a*(b*a^2)$. Any associative algebra $\mathfrak{A}$ (over a field of characteristic different than 2) becomes a Jordan algebra with the special Jordan product $a*b := \frac12 (ab+ba)$. 
Jordan algebras were originally studied in the context of quantum mechanics as a generalization of the space of observables of a quantum system~\cite{jordan1933}, but due to their close connection to symmetric cones, they have since then seen use in a variety of fields~\cite{faraut1994analysis,chu2011jordan,chu2017infinite}.

For a Jordan algebra $E$ we define the Jordan product operator $T_a:E\rightarrow E$ as $T_a(b) := a*b$. Two elements $a,b\in E$ are said to \Define{operator commute} when $T_aT_b = T_bT_a$. Using the commutativity of the product, this is easily seen to be equivalent to $a*(c*b) = (a*c)*b$ for all $c\in E$. Hence, all elements of a Jordan algebra operator commute if and only if the algebra is associative. It might then seem reasonable to think that the Jordan algebra generated by $a$ and $b$ (and the unit) is associative if and only if $a$ and $b$ operator commute. This is however not true (see for instance Remark 2.5.2 of~\cite{hanche1984jordan}). As such a correspondence between associative subalgebras and operator commutation is useful, there are a variety of results known about restrictions that do satisfy this equivalence. For instance if $b$ is \Define{idempotent}, $b^2 := b*b = b$, then $a$ and $b$ operator commute if and only if they generate an associative algebra~\cite[Lemma 2.5.5.]{hanche1984jordan}.

The original class of Jordan algebras to be studied were the \Define{Euclidean} Jordan algebras. These are finite-dimensional Jordan algebras over the real numbers that are \Define{formally real}: if $\sum_i^n a_i^2 = 0$ then $a_i = 0$ for all $i$. Equivalently, these are finite-dimensional algebras that are also real Hilbert spaces with the Jordan product being self-adjoint. For Euclidean Jordan algebras it is known that elements operator commute if and only if they generate an associative subalgebra\footnote{For instance, in~\cite[Lemma X.2.2]{faraut1994analysis} it is shown that elements operator commute if and only if they allow a simultaneous diagonalization. From that it easily follows that they generate an associative subalgebra.}.

Euclidean Jordan algebras were later generalized to an infinite-dimensional setting in the form of \Define{JB-algebras}. These are real Jordan algebras that are also Banach spaces with the Jordan product interacting suitably with the norm. Examples of JB-algebras include any norm-closed Jordan algebra $A\sse \mathfrak{A}_\sa$ where $\mathfrak{A}_\sa$ denotes the set of self-adjoint elements of a (unital) C$^*$-algebra. Such a Jordan subalgebra of a C$^*$-algebra is commonly referred to as a \Define{JC-algebra}. For JC-algebras it is also known that elements operator commute if and only they span an associative subalgebra\footnote{As far as the author is aware, this isn't explicitly stated anywhere, but it follows easily from~\cite[Lemma 5.1]{hanche1983structure}.}.

In this paper we will settle the question for JB-algebras. Letting $Q_a = 2T_a^2-T_{a^2}$ denote the \Define{quadratic map} of $a$ (that for associative algebras with the special Jordan product reduces to $Q_a(b) = aba$), we will prove the following.
\begin{theorem*}
    Let $A$ be a JB-algebra and $a,b\in A$ arbitrary. Then the following are equivalent.
    \begin{enumerate}[label=\alph*)]
        \item $a$ and $b$ operator commute.
        \item $a$ and $b$ generate an associative JB-algebra.
        \item $a$ and $b$ generate an associative JB-algebra of mutually operator commuting elements.
        \item $a$ and $a^2$ operator commute with $b$ and $b^2$.
    \end{enumerate}
    If one of $a$ and $b$ is positive then these statements are furthermore equivalent to $Q_a b^2 = Q_b a^2$.
\end{theorem*}

We prove this by resorting to the structure theory of JBW-algebras (JBW-algebras are to JB-algebras as von Neumann algebras are to C$^*$-algebras). Along the way we also give new, more algebraic, proofs of the statements in the above theorem for Euclidean Jordan algebras and JC-algebras.

As an application of our results, we show that the binary operation $a\mult b:= Q_{\sqrt{a}}b$ restricted to the unit interval of a JB-algebra satisfies the axioms of a \emph{sequential effect algebra}~\cite{gudder2002sequential}.

\section{Preliminaries}

\begin{definition}
    A \emph{Jordan algebra} $(E,*,1)$ over a field $\mathbb{F}$ is a vector space over $\mathbb{F}$ equipped with a unital commutative (not necessarily associative) bilinear operation $*: E\times E\rightarrow E$
    that satisfies the \Define{Jordan identity}: 
    \[(a*b)*(a*a) = a*(b*(a*a))\]
    We will refer to this operation as the \Define{Jordan product}.
\end{definition}

Though the definition of a Jordan algebra works for any field, as the theory of Jordan algebras is slightly different for fields of characteristic 2, we will assume that the field of the Jordan algebra has any characteristic other than 2.

\begin{example}
    Let $(A,\cdot,1)$ be an associative unital algebra over some field (not of characteristic 2).
    Then the operation $*:A\times A\rightarrow A$ defined by
    \[a*b := \frac12 (a\cdot b + b\cdot a)\]
    makes $(A,*,1)$ a Jordan algebra. We will refer to this operation as the \Define{special Jordan product} of the algebra.
    Any Jordan algebra that is isomorphic to a subset of an associative algebra equipped with this Jordan product is called \Define{special}.
\end{example}

To proceed we need some basic algebraic properties
of Jordan algebras, which are most conveniently expressed
with some additional notation.

\begin{definition}
Let~$E$ be a Jordan algebra.
\begin{enumerate}
\item
We write $a^0:=1$,\quad $a^1:= a$,\quad  $a^2:= a*a$,\quad  
        $a^3:= a*a^2$,\quad 
$a^4:=a*a^3$, \dots .
Note that since~$*$ is not associative
it's not a priori clear whether equations like  $a^4 = a^2 * a^2$ hold.
\item
Given~$a\in E$ we write $T_a: E\rightarrow E$ for the linear operator $T_a(b) := a*b$. We call these operators \Define{product operators}.
\item 
Given two linear operators $S,T\colon E\to E$
we write $[S,T]:= ST-TS$ for the commutator of~$S$ and~$T$.
\end{enumerate}
\end{definition}

Note that because the Jordan product is bilinear, we have $T_{a+\lambda b} = T_a + \lambda T_b$. The following identities can be found in any textbook on Jordan algebras, but as they show some fundamental results regarding operator commutativity we include them here. These are known as the \Define{linearized Jordan equations}:

\begin{lemma}\label{lem:jordan-equations}
Given a Jordan algebra~$E$,
and $a,b,c\in E$, we have
\begin{enumerate}[label=\alph*)]
\item \label{eq:jordan1}
$[T_a,T_{a^2}] = 0$
\item \label{eq:jordan2}
$[T_b,T_{a^2}] = 2[T_{a*b},T_{a}]$ 
\item \label{eq:jordan3}
$[T_a,T_{b*c}] + [T_b,T_{c*a}] + [T_c,T_{a*b}] = 0$;
\end{enumerate}
\end{lemma}
\begin{proof}~
\begin{enumerate}[label=\alph*)]
    \item The first equation, $[T_a,T_{a^2}]=0$, is just a reformulation of the Jordan identity:
        \begin{equation*}
                T_aT_{a^2}b\,\equiv\, a*(b*a^2) \ =\  (a*b)*a^2
            \,\equiv\, T_{a^2}T_a b.
        \end{equation*}
    
    \item Take the equality $[T_d,T_{d^2}] = 0$ and let $d=a\pm b$: $[T_{a\pm b},T_{(a\pm b)^2}]~=~0$. After expanding the terms using linearity we are left with
    $$[T_a, T_{a^2}] \pm [T_b, T_{b^2}] \pm \left([T_b, T_{a^2}] + 2 [T_a, T_{ab}]\right) +\left([T_a, T_{b^2}] + 2 [T_b, T_{ab}]\right) = 0.$$
    Subtracting the equation for $d=a+b$ from the equation for $d=a-b$ and dividing the result by 2 (here we use that the field is not of characteristic 2) we have the desired equation.

    \item Take the previous equation and replace $a$ by $a\pm c$. Using the same trick as before we arrive at the desired equation. \qedhere
\end{enumerate}
\end{proof}

\begin{proposition}
    Given a Jordan algebra~$E$ and $a,b,c \in E$ we have
    \begin{equation}\label{eq:jordan-normalise}
        T_{a*(b*c)} = T_aT_{b*c} + T_b T_{c*a} + T_c T_{a*b} - T_b T_a T_c - 
    T_c T_a T_b
    \end{equation}
\end{proposition}
\begin{proof}
    Apply the operators of \eqref{eq:jordan3} to an element $d$ and bring all the negative terms to the right to get
    $$a((bc)d) + b((ac)d) + c((ab)d) = (bc)(ad) + (ac)(bd) + (ab)(cd).$$
    Observe that the right-hand side is invariant under an interchange of $a$ and $d$ so that the left-hand side must be as well. This leads to the equality
    \begin{align*}
    a((bc)d) + b((ac)d) + c((ab)d) &= d((bc)a) + b((dc)a) + c((db)a) \\
    &= ((bc)a)d + b(a(cd)) + c(a(bd))
    \end{align*}
    where we have used the commutativity of the product to move $d$ to the end in the last equality.
    Translating this back into multiplication operators, using that this equality holds for all $d$, and bringing some terms to the other side then gives the desired equation.
\end{proof}

\subsection{Operator commutativity}

In this section we will collect and prove some results regarding operator commutativity in general Jordan algebras.

\begin{definition}
    Let $E$ be a Jordan algebra. We say $a,b\in E$ \Define{operator commute} when their Jordan product maps $T_a,T_b:E\rightarrow E$ commute, or equivalently when $a*(c*b) = (a*c)*b$ for all $c\in A$. We write $a\commu b$ to denote that $a$ and $b$ operator commute.
\end{definition}

\begin{definition}
    Let $E$ be a Jordan algebra, and let $S\sse E$ be some subset. We write $S'$ for the \Define{commutator} of $S$, defined as $S' := \{a\in E~;~ \forall s\in S: a\commu s\}$. 
\end{definition}

\begin{proposition}
	Let $E$ be a Jordan algebra and let $a\in E$. For any $n\in \N$, $T_{a^n}$ can be written as a polynomial in $T_{a^2}$ and $T_a$.
\end{proposition}
\begin{proof}
	We prove by induction. It is obviously true for $n=1,2$. Suppose it is true for all $k\leq n$. Then by Eq.~\eqref{eq:jordan-normalise} $T_{a^{n+1}} = T_{a*(a*a^{n-1})} = T_a T_{a^n} + T_a T_{a^n} + T_{a^{n-1}} T_{a^2} - T_a^2T_{a^{n-1}} - T_{a^{n-1}} T_a^2$. Expanding each of the $T_{a^n}$ and $T_{a^{n-1}}$ as polynomials of $T_a$ and $T_{a^2}$ finishes the proof.
\end{proof}

\begin{corollary}
	For any $a\in E$ and $n,m\in \N$, $a^n$ and $a^m$ operator commute and $a^n*a^m = a^{n+m}$.
\end{corollary}

\begin{corollary}
    For any $a\in E$, the Jordan algebra $J(a)$ generated by $a$ (consisting of polynomials in $a$) is associative.
\end{corollary}

\begin{corollary}
	For any $a,b\in E$, $b\in J(a)'$ iff $b\commu a,a^2$.
\end{corollary}

Let us remark that for general Jordan algebras, it could be that $a$ and $b$ operator commute, while $a^2$ does not operator commute with $b$. It is also possible that while $a$ and $b$ generate an associative subalgebra, they still do not operator commute. See for instance \cite[Remark 2.5.2]{hanche1984jordan} for explicit examples. 

With some more restrictions on $a$ and $b$ these two properties are however more related. We call an element $p\in E$ \Define{idempotent} when $p^2 = p*p = p$.

\begin{proposition}
	Let $E$ be a Jordan algebra with $a,p\in E$ where $p$ is idempotent. Then $a$ and $p$ operator commute if and only if $a$ and $p$ generate an associative subalgebra.
\end{proposition}
\begin{proof}
	See any textbook on Jordan algebras, e.g.~\cite[Lemma 2.5.5.]{hanche1984jordan}.
\end{proof}

A related result to this is the following:

\begin{proposition}\label{prop:idempotent-commutator}
	Let $E$ be a Jordan algebra, and let $p\in E$ be idempotent. Then $\{p\}'$ is a subalgebra of $E$.
\end{proposition}
\begin{proof}
	See e.g.~\cite[Lemma 1]{jacobson1989operator}.
\end{proof}

\begin{proposition}\label{prop:simple-operator-commutation}
	Let $E$ be a Jordan algebra with $a,b\in E$ and suppose $a\commu b, b^2$ and $b\commu a, a^2$. Then $a$ and $b$ span an associative subalgebra of mutually operator commuting elements.
\end{proposition}
\begin{proof}
	By repeatedly applying Eq.~\eqref{eq:jordan-normalise} any $T_p$ where $p$ is a polynomial in $a$ and $b$ can be reduced to a polynomial in $T_a$, $T_{a^2}, T_b, T_{b^2}$ and $T_{a*b}$, it hence remains to show that $a^2\commu b^2$, and that $a*b$ operator commutes with $a, a^2, b$ and $b^2$. By Lemma~\ref{lem:jordan-equations}.b) we already get $a*b\commu a,b$. With the same equation, but now taking $b:=b^2$, we see that $b^2 \commu a^2 \iff a*b^2\commu a$. Applying Eq.~\eqref{eq:jordan-normalise} to $T_{a*b^2}$ we see that it reduces to a polynomial in $T_b, T_a, T_{a*b}$ and $T_{b^2}$ and since $T_a$ commutes with them all, it commutes with $T_{a*b^2}$, and hence $a^2\commu b^2$.

    Taking \ref{lem:jordan-equations}.b) with $a:= a^2, b:= a, c:= b$ we get $[T_{a^2}, T_{a*b}] = - [T_a, T_{b*a^2}] - [T_b, T_{a^3}]$. As $b\commu a,a^2$ we also have $b\commu a^3$, and hence this last term dissapears. Expanding $T_{b*a^2}$ we see that $[T_a, T_{b*a^2}] = 0$ and hence indeed $[T_{a^2}, T_{a*b}] = 0$. Showing that $b^2\commu a*b$ follows entirely analogously.
\end{proof}

\begin{remark}
    As shown by the example in Ref.~\cite[Remark 2.5.2]{hanche1984jordan}, the conditions $a\commu b,b^2$ are necessary for the previous proposition to hold. It is unclear however whether the further assumption that also $b\commu a^2$ is necessary.
    In Ref.~\cite{jacobson1989operator}, to the authors knowledge the first paper dedicated to studying operator commutativity in Jordan algebras, it was shown that if $E$ is a Jordan algebra over a field of characteristic zero, and $B\sse E$ is a finite-dimensional semi-simple subalgebra, then $B'$ is a subalgebra itself. In particular, letting $B=J(a)$ for some $a\in E$, if $b\commu a,a^2$, then $b\in J(a)'$, and hence, since $J(a)'$ is a subalgebra, also $b^2\commu a$. We will see that when $E$ is a JB-algebra, that a similar property holds.
\end{remark}

\subsection{The quadratic product}

The Jordan product is generally not very well-behaved. In many circumstances it turns out to be better to work with the \emph{quadratic product}.

\begin{definition}
	Let $E$ be a Jordan algebra and let $a\in E$ be arbitrary. We define the \Define{quadratic product} of $a$ as $Q_a := 2T_a^2 - 2T_{a^2}$.
\end{definition}

Note that $Q_a 1 = a^2$ and $Q_1 = \id$.
For an associative algebra $A$ equipped with the standard Jordan product we easily calculate $Q_a b = aba$, hence the name quadratic product.

The quadratic product satisfies the following identity, sometimes referred as the \emph{fundamental equality of quadratic Jordan algebras}.
\begin{proposition}
	Let $E$ be a Jordan algebra, and $a,b\in E$ arbitrary. Then:
	\[Q_{Q_a b} = Q_aQ_bQ_a\]
\end{proposition}
Note that for special Jordan algebras this simply reduces to the evidently true equation $(aba)c(aba) = a(b(aca)b)a$.
Nevertheless, in the general case it is surprisingly hard to prove. In most textbooks it is proven only as a consequence of \emph{MacDonalds theorem} (which states that any polynomial equality in two variables that holds for special Jordan algebras, holds for all Jordan algebras)~\cite{hanche1984jordan,alfsen2012geometry,mccrimmon2006taste}. Although for the special case of finite-dimensional formally real Jordan algebras it can also be proven using analytic means~\cite{faraut1994analysis}. In Ref.~\cite{wetering2018algebraic} a semi-automated algebraic proof is given.

We will use the fundamental identity without further reference troughout this paper. Note that as an easy consequence we have $Q_{a^2} = Q_{Q_a 1} = Q_aQ_1Q_a = Q_a^2$.

\subsection{JB-algebras}

\begin{definition}\label{def:JB-algebra}
  Let $(A,*,1,\norm{\cdot})$ be a real Banach space (complete normed vector space) that is also a Jordan algebra. The space $A$ is called a \Define{JB-algebra} (Jordan-Banach) if the Jordan product $*$ satisfies for all $a,b\in A$:
  \begin{enumerate}[label=\alph*)]
    \item $\norm{a*b} \leq \norm{a}\norm{b}$.
    \item $\norm{a^2} = \norm{a}^2$.
    \item $\norm{a^2} \leq \norm{a^2 + b^2}$.
  \end{enumerate}
\end{definition}

\begin{remark}
	In \cite{hanche1984jordan}, JB-algebras are allowed to not have a unit. We will only deal with unital JB-algebras.
\end{remark}

\begin{proposition}[{\cite[Proposition 3.1.6]{hanche1984jordan}}]\label{prop:JB-define-as-OUS}
  Let $A$ be a JB-algebra. Then $A$ is an order unit space complete in the order-unit norm, and for all $a\in A$:
  \begin{equation*}
    -1 \leq a \leq 1 \implies 0\leq a^2\leq 1
  \end{equation*}
  Conversely, any complete order unit space with a Jordan product satisfying the above equation is a JB-algebra.
\end{proposition}

As a JB-algebra is an order unit space, it comes with a partial order $\leq$. The positive elements $a\geq 0$ in a JB-algebra precisely correspond with the squares: $a\geq 0 \iff \exists b: a=b^2$. Note that the Jordan product maps $T_a$ are not positive. I.e.\ if $a,b\geq 0$, then it is not necessarily the case that $a*b\geq 0$. The quadratic map is better behaved:

\begin{proposition}[{\cite[Proposition 3.3.6]{hanche1984jordan}}]
    Let $A$ be a JB-algebra, and let $a,b \in A$ be arbitrary. If $b\geq 0$ then $Q_a b\geq 0$.
\end{proposition}

\begin{example}
	Let $\mathfrak{A}$ be a unital C$^*$-algebra. Then the set of self-adjoint elements $\mathfrak{A}_{\sa}$ forms a JB-algebra with the Jordan product $a*b := \frac12(ab+ba)$.
\end{example}

\begin{definition}
	Let $A$ be a JB-algebra. We say $A$ is a \Define{JC}-algebra when there exists a C$^*$-algebra $\mathfrak{A}$ so that $A$ is isometrically isomorphic to a norm-closed subset of $\mathfrak{A}_{\sa}$.
\end{definition}

Note that the isometry $\phi: A\rightarrow \mathfrak{A}_{\sa}$ mapping a JC-algebra into its C$^*$-algebra is necessarily a Jordan homomorphism~\cite{wright1978isometries}. When studying JC-algebras as a Jordan algebra we can then hence without loss of generality assume it be a Jordan subalgebra of $\mathfrak{A}_{\sa}$.

\begin{definition}
	We call a Jordan algebra $E$ \Define{Euclidean} if it is also a finite-dimensional Hilbert space in such a way that the Jordan product operators are self-adjoint.
\end{definition}

\begin{proposition}
	A Euclidean Jordan algebra (EJA) is a JB-algebra. Conversely, any finite-dimensional JB-algebra is a Euclidean Jordan algebra.
\end{proposition}
\begin{proof}
	EJAs are precisely formally real finite-dimensional Jordan algebras~\cite[Proposition VIII.4.2]{faraut1994analysis}, and as shown in Ref.~\cite[Corollary 3.1.7]{hanche1984jordan}, any finite-dimensional formally real Jordan algebra is a JB-algebra. Conversely, by \cite[Corollary 3.3.8]{hanche1984jordan}, any JB-algebra is formally real, and hence any finite-dimensional JB-algebra is an EJA.
\end{proof}

Let $a\in A$ be an element of a JB-algebra. The Jordan algebra spanned by $a$ is associative, and as the Jordan product is continuous in the norm we can take the closure of the algebra, denoted as $C(a)$, and this algebra is still associative and furthermore consists of mutually operator commuting elements. An associative JB-algebra is easily seen to be equivalent to the self-adjoint part of a commutative C$^*$-algebra, and hence by the Gel'fand representation $C(a)$ is isomorphic to the set of continuous functions from some compact Hausdorff space to the real numbers. As a result we get a functional calculus for $a$, which in particular allows us to define a square root $\sqrt{a}$ that operator commutes with $a$. See \cite[Section 3.2]{hanche1984jordan} for more details.

\subsection{JBW-algebras}

Let $A$ be a JB-algebra. We call a subset $S\sse A$ \Define{directed} when for any two elements $s_1,s_2$ we can find a third element $s\in S$ such that $s_1\leq s$ and $s_2\leq s$. The subset is \Define{bounded} when there is some $a\in A$ such that for all $s\in S$, $s\leq a$. We call $A$ \Define{bounded directed-complete} when any non-empty bounded directed set has a supremum. We denote the supremum of a bounded directed set $S$ by $\bigvee S$.

\begin{definition}\label{def:normality}
    Let $f:A\rightarrow B$ be a positive map between JB-algebras. We call $f$ \Define{normal} when $f(\bigvee S) = \bigvee f(S)$ for any bounded directed set $S$ that has a supremum in $A$. If $B=\R$ then we call such a map a \Define{normal state}. We say the normal states are \Define{separating} when $f(a)=f(b)$ for all normal states $f$ implies that $a=b$.
\end{definition}

\begin{definition}\label{def:JBW-algebra}
  A JB-algebra $A$ is a \Define{JBW-algebra} when it is bounded directed-complete and has a separating set of normal states.
\end{definition}

\begin{example}
	Let $\mathfrak{A}$ be a von Neumann algebra, \ie a $C^*$-algebra that is bounded directed-complete and has a separating set of normal states~\cite{kadison1956operator}. Then its set of self-adjoint elements $\mathfrak{A}_{\sa}$ is a JBW-algebra with the standard Jordan product.
\end{example}

\begin{definition}
	A JBW-algebra $A$ will be called a \Define{JW-algebra} when it is isomorphic to an ultraweakly closed subset of the self-adjoint elements of a von Neumann algebra.
\end{definition}

Just as every C$^*$-algebra $\mathfrak{A}$ embeds into its double dual $\mathfrak{A}^{**}$ which is a von Neumann algebra, so does every JB-algebra $A$ embed into its double dual $A^{**}$ which is a JBW-algebra~\cite[Theorem 4.4.3]{hanche1984jordan}.

\begin{definition}
    Let $A$ be a JBW-algebra, denote by $V$ the vector space spanned by its normal states. The \Define{weak} topology of $A$ is the $\sigma(A,V)$ topology, \ie it is the weakest topology so that every map in $V$ is continuous. Concretely, a net $a_\alpha$ converges weakly to $a$ if $\omega(a_\alpha)$ converges to $a$ for every normal state $\omega$.
\end{definition}


We collect below a few well-known results regarding the weak and norm topology that we will use throughout the paper without further reference.

\begin{proposition}
	Let $A$ be a JBW-algebra and let $a\in A$ be an arbitrary element.
	\begin{enumerate}[label=\alph*)]
		\item Norm convergence implies weak convergence~\cite[Remark 4.1.3]{hanche1984jordan}.
		\item Let $S\sse A$ be a bounded directed subset. Then the net $(a_s)_{s\in S}$ converges weakly to $\bigwedge S$~\cite[Remark 4.1.3]{hanche1984jordan}.
		\item The operators $T_a$ and $Q_a$ are weakly continuous~\cite[Corollary 4.1.6]{hanche1984jordan}.
	\end{enumerate}
\end{proposition}

\begin{definition}
	Let $a\in A$ be an arbitrary element of a JBW-algebra $A$. Let $W(a)$ denote the weak closure of $C(a)$, \ie the JBW-algebra generated by $a$.
\end{definition}

Since $C(a)$ consists of mutually operator commuting elements and the Jordan product is weakly continuous, $W(a)$ also consists of mutually operator commuting elements and in particular, it is associative (see Ref.~\cite[Remark 4.1.10]{hanche1984jordan} for the details).

\begin{proposition}[{\cite[Proposition~4.2.3]{hanche1984jordan}}]
	Let $A$ be a JBW-algebra. The linear span of the idempotents of $A$ lies norm-dense in $A$.
\end{proposition}

In the next result we write JB(W)-algebra to note that the result holds for both JB-algebras and JBW-algebras.

\begin{proposition}
	Let $S\sse A$ be a Jordan subalgebra of the JB(W)-algebra $A$. Then $S'$ is a JB(W)-subalgebra.
\end{proposition}
\begin{proof}
	Let $a_n \in S$ converge to some $a$ (not necessarily in $S$) in the norm. Suppose $b\in S'$. Then for any $c\in A$, $T_bT_a c = b*(c*a) = \lim_n b*(c*a_n) = \lim_n T_bT_{a_n}c = \lim_n T_{a_n}T_b c = \lim_n (b*c)*a_n = (b*c)*a = T_aT_bc$. So any $b\in S'$ also commutes with everything in the norm closure of $S$. Without loss of generality we may then assume that $S$ is a JB-subalgebra.

	Suppose first that $A$ is a JBW-algebra. Then similarly we may assume that $S$ is weakly closed and hence is a JBW-algebra. But then $b\in S'$ iff $b\commu p$ for every idempotent $p\in S'$, and hence $S' = \bigcap_{p^2=p\in S}\ \{p\}'$. But as each $\{p\}'$ is a Jordan algebra by Proposition~\ref{prop:idempotent-commutator}, $S'$ will also be a Jordan algebra. It is easy to show that any norm or weak convergent net $b_n \in S'$ also operator commutes with all in $S'$, and hence $S'$ is a JBW-subalgebra.

	If $A$ is only a JB-algebra, then we can view $A$ as embedded into $A^{**}$. As $A$ lies weakly dense in $A^{**}$, the commutator $S'$ restricted to $A$ agrees whether taken as commutators in $A$ or $A^{**}$. Hence, if $b,b'\in S'$, then $b*b'\in S'$ as $S'$ is a subalgebra in $A^{**}$. So $S'$ is a Jordan subalgebra. It is easily seen to be norm-closed and hence is a JB-algebra.
\end{proof}


\noindent Using this proposition we can derive a stronger version of Proposition~\ref{prop:simple-operator-commutation} for JBW-algebras.

\begin{proposition}\label{prop:commuting-elements-span-associative-algebra}
	Let $A$ be a JBW-algebra, and $a,b\in A$ arbitrary. If $b\commu a$ and $b\commu a^2$, then $b^2\commu a$ and there is an associative subalgebra $B$ consisting of mutually operator commuting elements such that $W(a),W(b)\sse B$.
\end{proposition}
\begin{proof}
	If $b\commu a$ and $b\commu a^2$, then $b\in J(a)'$. By weak continuity of the Jordan product, $b \in W(a)'$. By the previous proposition $W(a)'$ is a subalgebra, so also $b^2 \in W(a)'$. Hence in particular $b^2\commu a$. Then by Proposition~\ref{prop:simple-operator-commutation}, $a$ and $b$ generate an associative Jordan algebra $S$ of mutually operator commuting elements. Let $B$ be the norm and weak closure of $S$, then $B$ has the desired properties.
\end{proof}

\subsection{Structure of JBW-algebras}

The JC-algebras are special Jordan algebras since they come from the associative product of the underlying C$^*$-algebra. The counterpart to that is an exceptional algebra that is not equal to some part of an associative algebra.

\begin{definition}
	Let $A$ be a JB-algebra. We call $A$ \Define{purely exceptional} when any Jordan homomorphism $\phi:A\rightarrow \mathfrak{A}_{\sa}$ onto a C$^*$-algebra $\mathfrak{A}$ is necessarily zero.
\end{definition}

\begin{theorem}[{\cite[Theorem 7.2.7]{hanche1984jordan}}]\label{thm:JBW-decomposition}
	Let $A$ be a JBW-algebra. Then there is a unique decomposition $A=A_{\text{ex}}\oplus A_{\text{sp}}$ where $A_{\text{sp}}$ is a JW-algebra and $A_{\text{ex}}$ is a purely exceptional JBW-algebra.
\end{theorem}

\begin{example}[{\cite{shultz1979normed}}]
	A compact Hausdorff space $X$ is called \Define{hyperstonean} when it is extremally disconnected and $C(X)$ is separated by normal states (\ie when $C(X)$ is a JBW-algebra). Let $E = M_3(\mathbb{O})_{\sa}$ denote the exceptional Albert algebra. Denote by $C(X,E)$ the set of continuous functions $f:X\rightarrow E$. Then $C(X,E)$ is a purely exceptional JBW-algebra with the Jordan product given pointwise by $(f*g)(x) = f(x)*g(x)$.
\end{example}

The above example is actually the only type of purely exceptional JBW-algebra, as the following result by Shultz shows.

\begin{theorem}[{\cite{shultz1979normed}}]
	Let $A$ be a purely exceptional JBW-algebra. Then there exists a hyperstonean compact Hausdorff space $X$, \ie that is extremally disconnected and where $C(X)$ is separated by normal states, such that $A\cong C(X,M_3(\mathbb{O}))$.
\end{theorem}

Since any JBW-algebra splits up into a direct sum of a JW-algebra and an algebra of the form $C(X,M_3(\mathbb{O})_{\sa})$, many questions regarding JBW-algebras can be settled by studying von Neumann algebras and Euclidean Jordan algebras (of which $M_3(\mathbb{O})_{\sa}$ is an example). Although direct proofs seem preferable, for some results it is not clear how to construct a direct proof, such as for the results of this paper.

\section{Main results}

We are now ready to start proving our main results. We first study operator commutation in JC-algebras. Then we study it in Euclidean Jordan algebras. Then we combine these results with the structure theory of JBW-algebras of the previous section to derive some consequences for general JBW-algebras. Finally, we generalize this to JB-algebras.

Recall that an element $a$ in a C$^*$-algebra is called \Define{normal} when $aa^* = a^*a$. This should not be confused with maps between JBW-algebras preserving suprema that are also called normal. We will not use the latter definition of `normal' in this section. The following theorem regarding normal elements will be our primary calculational tool in this section.

\begin{theorem}[Fuglede-Putnam-Rosenblum]
  Let $m,n,a\in \mathfrak{A}$ be elements of a C$^*$-algebra, with $m$ and $n$ normal and $ma = an$. Then $m^*a = an^*$.
\end{theorem}

\begin{proposition}\label{prop:JW-operator-commutation}
  Let $A\sse \mathfrak{A}$ be a JC-algebra acting on the C*-algebra $\mathfrak{A}$. Elements of $A$ operator commute if and only if they commute as elements of $\mathfrak{A}$: $T_aT_b = T_bT_a \iff ab=ba$.
\end{proposition}
\begin{proof}
  In~\cite[Lemma 5.1]{hanche1983structure} they prove this result using representation theory. We present here a more algebraic proof.
  Obviously, when $ab=ba$, we have $T_aT_b = T_bT_a$. So let us prove the converse direction.
  For $a,b,c\in A$ we easily calculate:
  \begin{equation}\label{eq:proof-operator-commutation-in-vNA}
    (T_aT_b-T_bT_a)c = \frac14 ((ab-ba)c - c(ab-ba))
  \end{equation}
  Hence, when $a$ and $b$ operator commute we have $(ab-ba)c = c(ab-ba)$ for all $c\in A$. 
  In particular, take $c=a$ and $c=b$ to get:
  \[2aba = ba^2 + a^2b \qquad 2bab = ab^2 + b^2a\]
  Now multiply the first equation by $b$ on the right and the second with $a$ on the left:
  \[2abab = ba^2b +a^2b^2 \qquad 2abab = a^2b^2 + ab^2a\]
  As the left-hand sides now agree, we can combine the equations to get $ba^2b = ab^2a$. 
  This equation shows that $ab$ is normal: $(ab)(ab)^* = ab^2a = ba^2b = (ab)^*(ab)$, and hence by the Fuglede-Putnam-Rosenblum theorem, since $(ab)a = a(ba)$ we also have $ba^2 = (ab)^*a = a(ba)^* = a^2b$ and so $b$ and $a^2$ commute.

  Recall that we had the equation $2aba = ba^2 + a^2b$. Using the operator commutation we get $aba = a^2b$. We claim that this equality in combination with $a^2b=ba^2$ implies that $ba = ab$. As the C$^*$-algebra $\mathfrak{A}$ embeds into the von Neumann algebra $\mathfrak{A}^{**}$, and these equations continue to hold in this von Neumann algebra it suffices to show this implication for von Neumann algebras.

  Let $r(a)$ denote the \emph{range-projection} of $a$ in $\mathfrak{A}^{**}$, \ie the smallest projection such that $r(a)a = a$. Recall that $r(a^2) = r(a)$.
  By applying the \emph{approximate pseudoinverse} of $a$ (see Ref.~\cite[Section~3.4.1]{bramthesis} for more details) to the left of $aba = a^2b$ we get $r(a)ba = r(a)ab = ab$. As $b$ commutes with $a^2$, it commutes with $r(a^2) = r(a)$, and hence $ab=r(a)ba = br(a)a= ba$, and we are done.
\end{proof}

\begin{proposition}\label{prop:JW-quadratic-commutation}
    Let $A\sse \mathfrak{A}$ be a JC-algebra acting on the C$^*$-algebra $\mathfrak{A}$. Let $a,b\in A$ with at least one of $a$ and $b$ positive. Then the following are equivalent.
    \begin{enumerate}[label=\alph*)]
        \item $Q_aQ_b =Q_bQ_a$.
        \item $Q_ab^2 = Q_b a^2$.
        \item $a$ and $b$ operator commute.
    \end{enumerate}
\end{proposition}
\begin{proof}
    a) to b) is trivial. For c) to a) we note that $a$ and $b$ operator commute if $ab=ba$ in $\mathfrak{A}$, and hence also $a$ and $a^2$ operator commute with $b$ and $b^2$. The result then follows by the definition of $Q_a$ and $Q_b$ in terms of $T_a, T_{a^2}, T_b$ and $T_{b^2}$. It remains to prove b) to c). Suppose $Q_ab^2 = Q_ba^2$. Written in terms of the associative product of $\mathfrak{A}$ this becomes $ab^2a = ba^2b$. Hence, $(ab)$ is normal. Without loss of generality, assume that $a$ is positive. Since $(ab)a = a(ba)$, by the Fuglede-Putnam-Rosenblum theorem: $ba^2 = (ab)^*a = a(ba)^* = a^2 b$, so that $b$ and $a^2$ commute. By positivity of $a$, $\sqrt{a^2} = a$. Since $\sqrt{a^2}$ lies in the bicommutant of $a^2$ we then see that $b$ also commutes with $a$ in $\mathfrak{A}$ and hence they operator commute in $A$.
    %
\end{proof}

\begin{remark}
	Of course c) implies a) and a) implies b) regardless of positivity of $a$ and $b$, but for the other implications, the requirement that at least one of $a$ and $b$ is positive is necessary. Take for instance any non-commuting $a$ and $b$ satisfying $a^2=b^2 = 1$ such as $a=\ket{+}\bra{+}-\ket{-}\bra{-}$ and $b=\ket{0}\bra{0}-\ket{1}\bra{1}$. Then b) holds, but a) and c) do not. Keeping $b$ the same, but letting $a=\ket{0}\bra{1}+\ket{1}\bra{0}$ we get a) but not c). 
\end{remark}

For our next results we will need the following powerful theorem which is shown in Theorem 7.2.5 of Ref.~\cite{hanche1984jordan}.
\begin{theorem}[Shirshov-Cohn for JB-algebras]
	A JB-algebra generated by two elements (and possibly the unit) is a JC-algebra.
\end{theorem}

Note that the proof given in Ref.~\cite{hanche1984jordan} applies equally well for JBW-algebras:

\begin{theorem}[Shirshov-Cohn for JBW-algebras]
	A JBW-algebra generated by two elements (and possibly the unit) is a JW-algebra.
\end{theorem}

\begin{proposition}\label{prop:JBW-two-elements-associative}
	Let $A$ be a JB(W)-algebra, and suppose $a,b \in A$ either operator commute or at least one of them is positive and they satisfy $Q_ab^2 = Q_ba^2$. Then the JB(W)-algebra spanned by $a$ and $b$ (and possibly the unit) is associative.
\end{proposition}
\begin{proof}
	Let $B$ denote the JB(W)-algebra spanned by $a$ and $b$. By the Shirshov-Cohn theorem, $B$ is a JC-algebra (respectively JW-algebra). Let $\mathfrak{B}$ denote the C$^*$-algebra $B$ acts on.

	If $a$ and $b$ operator commute in $A$, then they of course also operator commute in $B$ and hence by Proposition~\ref{prop:JW-operator-commutation} $ab=ba$ in $\mathfrak{B}$. Similarly, but by Proposition~\ref{prop:JW-quadratic-commutation}, if $a$ and $b$ are positive and satisfy $Q_ab^2 = Q_ba^2$, they operator commute (inside of $B$), and hence also $ab=ba$ in $\mathfrak{B}$.

	In both cases we then also have $a^2b = ba^2$ so that $a^2$ operator commutes with $b$ when restricted to $B$. By Proposition~\ref{prop:commuting-elements-span-associative-algebra}, there is an associative JB(W)-subalgebra $B'$ of $B$ containing both $a$ and $b$. But as $B$ is already the smallest JB(W)-algebra generated by $a$ and $b$ we necessarily have $B' = B$.
\end{proof}

\begin{corollary}\label{cor:JBW-assoc-quadratic}
	Let $a,b\in A$ in a JB-algebra with at least one of $a$ and $b$ being positive. If $Q_ab^2 = Q_ba^2$, then $Q_ab^2 = a^2*b^2$.
\end{corollary}
\begin{proof}
	By the previous proposition, $a$ and $b$ span an associative algebra, and the statement is true for associative algebras.
\end{proof}

\begin{proposition}
    Let $E$ be a Euclidean Jordan algebra with $a,b \in E$ where at least one of $a$ and $b$ is positive. Then the following are equivalent.
    \begin{enumerate}[label=\alph*)]
    	\item $Q_aQ_b =Q_bQ_a$.
        \item $Q_ab^2 = Q_b a^2$.
        \item $b$ and $b^2$ operator commute with $a$ and $a^2$.
    \end{enumerate}
\end{proposition}
\begin{proof}
	a) to b) and c) to a) are trivial, hence it suffices to prove b) to c). Nevertheless, it will be useful to first prove b) to a).

	So assume that $Q_ab^2 = Q_ba^2$. 
	Since $E$ is a Euclidean Jordan algebra, it is a Hilbert space, and hence the space of bounded operators on $E$, $B(E)$, is a C$^*$-algebra. Recall that $Q_a$ for any $a\in E$ is a self-adjoint operator and hence, if $a$ is positive, $Q_a = Q_{\sqrt{a}}^2$ is a positive operator in the Hilbert space sense, \ie positive in $B(E)$.
	By the fundamental equality we have $Q_aQ_b^2Q_a = Q_{Q_ab^2} = Q_{Q_ba^2} = Q_bQ_a^2Q_b$ so that $Q_aQ_b$ is normal as an element of $B(E)$. Hence, analogously to the proof of Proposition~\ref{prop:JW-quadratic-commutation}, we can use the Fuglede-Putnam-Rosenblum theorem to conclude that $Q_a Q_b = Q_bQ_a$ which proves b) to a).

	Now for b) to c), since $Q_ab^2 = Q_ba^2$, by Proposition~\ref{prop:JBW-two-elements-associative}, $a$, $b$ and $1$ span an associative JBW-algebra $B$ which in this case is an EJA. 
	Then $a+1$ also lies in $B$, and hence $Q_{a+1}b^2 = (a+1)^2*b^2 = Q_b (a+1)^2$. Using b) to a) we then see that $Q_{a+1}Q_b = Q_bQ_{a+1}$. As also $Q_aQ_b = Q_bQ_a$ and $T_a = Q_{a+1} - Q_a - \id$ we then necessarily have $T_aQ_b = Q_bT_a$. In a similar way we also get $T_aQ_{b+1} = Q_{b+1}T_a$ and hence $T_aT_b = T_bT_a$. As $a^2$ is also part of the same associative JBW-algebra, we can repeat the argument with $a^2$ instead of $a$ to see that $T_{a^2}T_b = T_bT_{a^2}$. Similarly, we can take $b^2$ instead of $b$.
\end{proof}

\begin{lemma}\label{lem:exceptional-quadratic-commutation}
	Let $A$ be a purely exceptional JBW-algebra with $a,b \in A$ where at least one of $a$ and $b$ is positive. Then the following are equivalent.
	\begin{enumerate}[label=\alph*)]
    	\item $Q_aQ_b =Q_bQ_a$.
        \item $Q_ab^2 = Q_b a^2$.
        \item $b$ and $b^2$ operator commute with $a$ and $a^2$.
    \end{enumerate}
\end{lemma}
\begin{proof}
	a) to b) and c) to a) are trivial, so only b) to c) remains. Hence, suppose that $Q_ab^2 = Q_ba^2$.

	Since $A$ is a purely exceptional, $A\cong C(X,E)$ where $X$ is a hyperstonean compact Hausdorff space, and $E=M_3(\mathbb{O})_{\text{sa}}$ is the exceptional Albert algebra and in particular, $E$ is an EJA. Let $f,g:X\rightarrow E$ denote the functions corresponding to $a$ and $b$. Then for every $x\in X$, $Q_{f(x)}g(x)^2 = (Q_f g^2)(x) = (Q_g f^2)(x) = Q_{g(x)} f(x)^2$. As these are elements of an EJA, we use the previous proposition to see that $T_{g(x)}$ and $T_{g(x)^2}$ operator commute with $T_{f(x)}$ and $T_{f(x)^2}$ for all $x\in X$. We wish to conclude from this that $T_g$ and $T_{g^2}$ operator commute with $T_f$ and $T_{f^2}$. So let $h:X\rightarrow E$ be any other function, then we should have for every $x\in X$, $(T_fT_gh)(x) = (T_gT_fh)(x)$ (and similarly for $f^2$ and $g^2$), but as $(T_fT_gh)(x) = T_{f(x)}T_{g(x)}h(x)$ this of course directly follows.
\end{proof}

\begin{proposition}\label{prop:JBW-quadratic-commute}
	Let $A$ be a JBW-algebra with $a,b \in A$ where at least one of $a$ and $b$ is positive. Then the following are equivalent.
	\begin{enumerate}[label=\alph*)]
    	\item $Q_aQ_b =Q_bQ_a$.
        \item $Q_ab^2 = Q_b a^2$.
        \item $b$ and $b^2$ operator commute with $a$ and $a^2$.
    \end{enumerate}
\end{proposition}
\begin{proof}
	Write $A=A_1\oplus A_2$ where $A_1$ is a JW-algebra, and $A_2$ is a purely exceptional JBW-algebra. If $Q_ab^2 = Q_ba^2$, then this equation also holds with $a$ and $b$ restricted to $A_1$ or $A_2$. By Proposition~\ref{prop:JW-quadratic-commutation} and Lemma~\ref{lem:exceptional-quadratic-commutation} the desired result then follows.
\end{proof}

\begin{theorem}\label{thm:JBW-assoc-iff-commute}
    Let $A$ be a JBW-algebra and $a,b\in A$ arbitrary. Then the following are equivalent.
    \begin{enumerate}[label=\alph*)]
        \item $a$ and $b$ operator commute.
        \item $a$ and $b$ generate an associative JBW-algebra.
        \item $a$ and $b$ generate an associative JBW-algebra of mutually operator commuting elements.
        \item $a$ and $a^2$ operator commute with $b$ and $b^2$.
    \end{enumerate}
    If one of $a$ and $b$ is positive then these statements are furthermore equivalent to $Q_a b^2 = Q_b a^2$.
\end{theorem}
\begin{proof}
	a) to b) follows by Proposition~\ref{prop:JBW-two-elements-associative}. c) to d) follows because $a^2$ and $b^2$ are part of the associative algebra of mutually operator commuting elements. d) to a) is of course trivial. It remains to show b) to c).

    So suppose $a$ and $b$  generate an associative JBW-algebra $B$. Let $c,d\in B$ be positive. By associativity $Q_c d^2 = c^2*d^2 = Q_d c^2$, and hence by Proposition~\ref{prop:JBW-quadratic-commute} $c$ and $d$ operator commute. But the positive elements of course span $B$, so we are done.

    Now suppose one of $a$ and $b$ is positive. If they span an associative algebra, then of course $Q_a b^2 = a^2*b^2 = Q_b a^2$. For the converse direction we again refer to Proposition~\ref{prop:JBW-quadratic-commute}.
\end{proof}

\begin{theorem}\label{thm:JB-assoc-iff-commute}
	Let $A$ be a JB-algebra and $a,b\in A$ arbitrary. Then the following are equivalent.
    \begin{enumerate}[label=\alph*)]
        \item $a$ and $b$ operator commute.
        \item $a$ and $b$ generate an associative JB-algebra.
        \item $a$ and $b$ generate an associative JB-algebra of mutually operator commuting elements.
        \item $a$ and $a^2$ operator commute with $b$ and $b^2$.
    \end{enumerate}
    If one of $a$ and $b$ is positive then these statements are furthermore equivalent to $Q_a b^2 = Q_b a^2$.
\end{theorem}
\begin{proof}
	Again, a) to b), c) to d) and d) to a) are trivial. For proving b) to c) suppose $a$ and $b$ generate an associative JB-algebra $B$. Note that the JB-algebra $A$ embeds into the JBW-algebra $A^{**}$, and hence $B$ extends to an associative JBW-algebra in $A^{**}$. The result then follows in the same way as in the previous theorem.

	To show that $Q_a b^2 =Q_b a^2$ for one of $a$ and $b$ positive implies that $a$ and $b$ operator commute we again note that the equality continues to hold in $A^{**}$, and hence by the previous theorem, $a$ and $b$ operator commute in $A^{**}$, and thus in $A$.
\end{proof}

\section{The sequential product}

Let us use our results regarding operator commutation to prove that the unit interval in a JB-algebra is a sequential effect algebra as defined by Gudder and Greechie. Let $A$ be a JB-algebra. Write $[0,1]_A$ for its set of \Define{effects}, \ie the elements $a\in A$ with $0\leq a\leq 1$. The set of effects in a JB-algebra forms a (convex) effect algebra. For an effect $a\in [0,1]_A$ we write $a^\perp := 1-a$.

\begin{lemma}[{\cite[Lemma~1.26]{alfsen2012geometry}}]\label{lem:quadratic-is-zero}
	Let $a$ and $b$ be positive elements in a JB-algebra. Then $Q_a b =0$ iff $Q_b a = 0$, and in that case $a*b = 0$.
\end{lemma}

\begin{theorem}
	Let $A$ be a JB-algebra. Define the operation $\&:[0,1]_A\times [0,1]_A\rightarrow [0,1]_A$ by $a\mult b := Q_{\sqrt{a}}b$. Then $\&$ satisfies all the axioms of Ref.~\cite{gudder2002sequential}:
	\begin{enumerate}[label=\alph*)]
		\item $a\mult (b+c) = a\mult b + a\mult c$.
		\item $1\mult a = a$.
		\item If $a\mult b = 0$, then also $b\mult a = 0$.
		\item If $a\mult b = b\mult a$, then $a\mult b^\perp = b^\perp \mult a$, and $a\mult (b\mult c) = (a\mult b)\mult c$. 
		\item If $a\mult b = b\mult a$ and $a\mult c = c\mult a$, then $a\mult (b+c) = (b+c)\mult a$ and $a\mult (b\mult c) = (b\mult c)\mult a$.
	\end{enumerate}
\end{theorem}
\begin{proof}
	Points a) and b) are trivial.
	
	For c), if $a\mult b = 0 = Q_{\sqrt{a}} b$, then $Q_b \sqrt{a} = 0$ by Lemma~\ref{lem:quadratic-is-zero}. Hence also $Q_b a^2 \leq Q_b a \leq Q_b \sqrt{a} = 0$. Applying Lemma~\ref{lem:quadratic-is-zero} again gives $Q_a b = 0$, so that also $Q_a b^2 = 0$. But then $Q_a b^2 = 0 = Q_b a^2$, and hence by Theorem~\ref{thm:JB-assoc-iff-commute}, $a$ and $b$ span a JBW-algebra of mutually commuting elements. This algebra necessarily contains $\sqrt{b}$ and hence $b\mult a = Q_{\sqrt{b}} \sqrt{a}^2 = a*b = 0$ by Corollary~\ref{cor:JBW-assoc-quadratic} and Lemma~\ref{lem:quadratic-is-zero}.

	Note that if $a\mult b = b\mult a$, then by definition $Q_{\sqrt{a}} \sqrt{b}^2 = Q_{\sqrt{b}}\sqrt{a}^2$, so that by Theorem~\ref{thm:JB-assoc-iff-commute}, $\sqrt{a}$ and $\sqrt{b}$ span an associative algebra of mutually commuting elements, and hence $Q_{\sqrt{a}}Q_{\sqrt{b}} = Q_{\sqrt{b}}Q_{\sqrt{a}}$, and $a\mult b = Q_{\sqrt{a}}\sqrt{b}^2 = \sqrt{a}^2*\sqrt{b}^2 = a*b$. Furthermore, as this algebra also contains $a^{1/4} = \sqrt{\sqrt{a}}$, $Q_{a^{1/4}}$ commutes with $Q_{\sqrt{b}}$, and hence $a^{1/4}$ and $\sqrt{b}$ also generate an associative JBW-algebra.

	For point d) suppose that $a\mult b = b\mult a$. Then $\sqrt{a}$ and $\sqrt{b}$ span an associative JBW-algebra containing $1$, and this algebra hence also contains $b^\perp = 1-b$ and $\sqrt{b^\perp}$. As a result $a\mult b^\perp = Q_{\sqrt{a}}b^\perp = a*b^\perp = Q_{\sqrt{b^\perp}} a = b^\perp \mult a$. Furthermore, as the JBW-algebra spanned by $a^{1/4}$ and $\sqrt{b}$ is associative we easily verify that $Q_{a^{1/4}}\sqrt{b} = \sqrt{Q_{\sqrt{a}}b}$ and then calculate
	\[a\mult (b\mult c) = Q_{\sqrt{a}} Q_{\sqrt{b}} c = Q_{a^{1/4}}Q_{\sqrt{b}} Q_{a^{1/4}} c = Q_{Q_{a^{1/4}}\sqrt{b}} c = Q_{\sqrt{Q_{\sqrt{a}}b}} c = (a\mult b)\mult c.\]

	Finally, for point e) suppose that $a\mult b = b\mult a$ and $a\mult c = c\mult a$.
	Then $a$ operator commutes with $b$ and $c$ and hence with $b+c$. By Theorem~\ref{thm:JB-assoc-iff-commute}, $a$ and $b+c$ then generate an associative algebra, and hence $(b+c)\mult a = Q_{\sqrt{b+c}} a = (b+c)*a = Q_{\sqrt{a}}(b+c) = a\mult (b+c)$ as desired.
	And at last, as $[Q_{\sqrt{a}},Q_{\sqrt{b}}] = 0$ and $[Q_{\sqrt{a}},Q_{\sqrt{c}}] = 0$, we calculate
	\[Q_{Q_{\sqrt{b}}c} Q_{\sqrt{a}} = Q_{\sqrt{b}}Q_{\sqrt{c}}^2Q_{\sqrt{b}} Q_{\sqrt{a}} = Q_{\sqrt{a}} Q_{\sqrt{b}}Q_{\sqrt{c}}^2Q_{\sqrt{b}} = Q_a Q_{Q_{\sqrt{b}}c}.\]
	Then $a$ and $Q_{\sqrt{b}} c$ generate an associative algebra so that we can finally verify that indeed $a\mult (b\mult c) = (b\mult c)\mult a$.
\end{proof}

In Ref.~\cite{wetering2018sequential} a further assumption of norm-continuity on the sequential product was assumed. This is also satisfied by the sequential product on JB-algebras:

\begin{proposition}
	Let $A$ be a JB-algebra with $\&$ as defined above. Then the map $a\mapsto a\mult b$ is continuous in the norm.
\end{proposition}
\begin{proof}
	The map $a\mapsto a\mult b$ is given by the composition of $a\mapsto \sqrt{a}$ and $c\mapsto Q_c b$. The first is an application of the functional calculus and hence is continuous, while the second follows from the continuity of the Jordan product.
\end{proof}

In Ref.~\cite{gudder2002sequential} a notion of $\sigma$-sequential effect algebra was also introduced. In such an effect algebra, suprema of increasing sequences exist and are compatible in a suitable way with the sequential product. We will show now that the unit interval of a JBW-algebra satisfies an even stronger condition related to suprema of arbitrary directed sets.

\begin{proposition}
	Let $A$ be a JBW-algebra with $\&$ defined as above. Let $a\in [0,1]_A$ and $S\sse [0,1]_A$ be a directed subset. Then:
	\begin{itemize}
		\item $a\mult \bigvee S = \bigvee a\mult S$ (\ie $\&$ is normal in the second argument).
		\item If for all $s\in S$, $a\mult s = s\mult a$, then $a\mult \bigvee S = (\bigvee S)\mult a$.
	\end{itemize}
\end{proposition}
\begin{proof}
	The first point follows immediately by normality of the quadratic product. For the second point, if $a\mult s = s\mult a$ for all $s\in S$, then $a$ operator commutes with all these $s$. By the weak continuity of the Jordan product, $a$ then also operator commutes with $\bigvee S$, and the desired result follows.
\end{proof}

\textbf{Acknowledgements}: The author would like to thank Bas and Bram Westerbaan for many discussions regarding Jordan algebras. The author is supported in part by AFOSR grant FA2386-18-1-4028..

\bibliographystyle{eptcs}
\bibliography{../bibliography}

\end{document}